\documentclass[12pt,a4paper,reqno]{amsart}
\usepackage[hmargin=2.5cm,bmargin=2.5cm,tmargin=3cm]{geometry}
\usepackage{enumerate}
\usepackage{orcidlink}
\usepackage{amsmath,amsthm,amssymb,amsfonts,latexsym}
\usepackage[gen]{eurosym}
\usepackage[german,english]{babel}

\usepackage{latexsym}

\usepackage{tikz}					
\usetikzlibrary{shapes,arrows,decorations.pathmorphing,backgrounds,fit,positioning,shapes.symbols,chains,graphs,graphs.standard}

\newcommand{\cD}{\mathcal{D}}

\newtheorem{theorem}{Theorem}

\newtheorem{proposition}[theorem]{Proposition}

\newtheorem{remark}[theorem]{Remark}
\newtheorem{corollary}[theorem]{Corollary}

\newtheorem{example}[theorem]{Example}
\newcommand{\be}{\begin{equation}}
\newcommand{\ee}{\end{equation}}
\newcommand{\bea}{\begin{eqnarray*}}
	\newcommand{\eea}{\end{eqnarray*}}
\newcommand{\beq}{\begin{eqnarray}}
\newcommand{\eeq}{\end{eqnarray}}

\newcommand{\Graph}{\mathcal G}
\newcommand{\mV}{\mathsf V}

\newcommand{\mv}{\mathsf v}
\newcommand{\mE}{\mathsf E}
\newcommand{\me}{\mathsf e}

\usepackage{todonotes}

\title[A modified local Weyl law and spectral comparison]{A modified local Weyl law and spectral comparison results for $\delta'$-coupling conditions} 

\subjclass[2010]{}


\author[P.~Bifulco]{Patrizio Bifulco\orcidlink{0009-0004-0628-374X}}
\author[J.~Kerner]{Joachim Kerner\orcidlink{0000-0003-0638-4183}}

\address{Lehrgebiet Analysis, Fakult\"at Mathematik und Informatik, Fern\-Universit\"at in Hagen, D-58084 Hagen, Germany}
\email{patrizio.bifulco@fernuni-hagen.de}

\address{Lehrgebiet Analysis, Fakult\"at Mathematik und Informatik, Fern\-Universit\"at in Hagen, D-58084 Hagen, Germany}
\email{joachim.kerner@fernuni-hagen.de}

\date{\today}

\thanks{
}

\begin{document}
	
	\begin{abstract} We study Schrödinger operators on compact finite metric graphs subject to $\delta'$-coupling conditions. Based on a novel modified local Weyl law, we derive an explicit expression for the limiting mean eigenvalue distance of two different self-adjoint realisations on a given graph. Furthermore, using this spectral comparison result, we also study the limiting mean eigenvalue distance comparing $\delta'$-coupling conditions to so-called anti-Kirchhoff conditions, showing divergence and thereby confirming a numerical observation in \cite{BandSchanzSofer}. 
	\end{abstract}
	
\maketitle

	\section{Introduction}
	
To study spectral properties of Schrödinger-type operators on domains in Euclidean space, manifolds or metric graphs is an old and fascinating endeavour in mathematics and mathematical physics. In many cases, the spectrum of the considered operator is purely discrete and one then aims at understanding the dependence of the eigenvalues on the potential and the geometrical features of the object the operator is acting on. Along this line of reasoning, metric graphs have become a popular model in spectral theory in the last decades since they interpolate between one- and higher-dimensional aspects; encountered, for example, in the study of manifolds. In other words, although metric graphs are locally one-dimensional objects, their global behaviour is more involved due to their complex topology. At this point, one should also mention that quantum graphs have also become important models in other areas such as -- for example -- quantum chaos~\cite{KS} and quantum chemistry~\cite{RS}. For a general introduction to quantum graphs we refer to~\cite{BerkolaikoBook,Mugnolob,KurasovBook} and references therein. 

The goal of this article is to establish spectral comparison results for Schrödinger-type operators on finite metric graphs and subject to so-called \emph{$\delta'$-coupling conditions}. Here, one should note that the prototypical coupling conditions on metric graphs are so-called \emph{$\delta$-coupling conditions} which generalize the well-known notion of a Dirac-$\delta$ potential on the real line; recall that, for such a singular potential, the function remains continuous but the derivative has a jump at the position of the potential. For $\delta'$-coupling conditions, this is reversed in the sense that now the derivative is continuous but the function itself is not! A special but important case of $\delta'$-coupling conditions are the so-called \emph{anti-Kirchhoff conditions} whose counterpart are \emph{Kirchhoff} (or \emph{standard}) coupling conditions; see \cite{RS23} for an overview.   

More specifically, our aim is to derive an explicit expression for the mean eigenvalue difference of two (different) Schrödinger operators on a given metric graph subject to $\delta^{\prime}$-coupling conditions (see Theorem~\ref{MainResult1}). We then use this expression in combination with other results to compare the spectrum of Schrödinger operators subject to $\delta^{\prime}$- and anti-Kirchhoff coupling conditions (see Theorem~\ref{MainResultII}); by doing this, we provide a rigorous explanation of a numerical observation made in \cite{BandSchanzSofer}. Note that, for Schrödinger operators with $\delta$-coupling conditions, an explicit expression for the mean eigenvalue difference was recently established in \cite{BK23} (see also \cite{BandSchanzSofer} for comparable results) using properties of the heat-kernel on graphs~\cite{BER}. For example, the heat-kernel asymptotics for small times can be used to derive a so-called \emph{local Weyl law} which is then used to establish the main spectral comparison result. On graphs and using heat-kernel techniques, a local Weyl law was established in \cite{BHJ} for standard coupling conditions and then generalized in \cite{BK23} to Schrödinger operators with general coupling conditions. Nevertheless, in order to treat $\delta^{\prime}$-coupling conditions, a modified local Weyl law has to be established (see Proposition~\ref{prop:localweyllaw}); this is due to the specific structure of the quadratic form associated to the considered Schrödinger operator. At this point, let us also refer to \cite{BKInfinite} where a modified local Weyl law has been derived for a certain infinite quantum graph. In the infinite setting, however, the modification was not due to a certain structure of the quadratic form but rather due to a non-standard Weyl law for the eigenvalue counting function. 

Let us also mention that similar spectral comparison results can be derived in different settings. For example, the paper \cite{BK23} had been motivated by \cite{RWY} (which itself was inspired by \cite{RR}) where the authors compare the spectrum of Laplacians on domains in $\mathbb{R}^2$ subject to Neumann and Robin boundary conditions, respectively. In the same manner, similar results can be established also for manifolds. 
 
This paper is organized as follows: In Section~\ref{Setting} we describe the setting and introduce the self-adjoint operators of interest. In Section~\ref{MainSpectral} we collect our main spectral comparison results whose proofs are then established in Section~\ref{ProofSection}. In Section~\ref{SomeAdditionalSpectral} we provide some additional spectral comparison results for bipartite graphs. Additionally, in Appendix~\ref{sec:general-potential-delta} we derive small-time  asymptotics for the heat-kernel for a larger class of potentials (compared to the one considered in~\cite{BER}) which then leads to a generalization of results obtained in \cite{BK23,BKAmba,BKInfinite}. 

\section{The setting}\label{Setting}
 Let $\Graph=\Graph(\mE,\mV)$ be a connected finite metric graph with (finite) edge set $\mE$ and (finite) vertex set $\mV$. To each edge $\me \in \mE$ we associate an interval $(0,\ell_\me)$ of finite length, i.e., we have $0 < \ell_\me < \infty$; $\ell_\me$ denotes the \emph{edge length}. The number of edges connected to a vertex $\mv \in \mV$ is called the \textit{degree} and shall be denoted by $\deg(\mv) \in \mathbb{N}$; more precisely,
  \begin{align*}
  \deg(\mv) := \#\{ \me \in \mE \: : \: \me \text{ is incident to } \mv\} \quad \text{for }\mv \in \mV \ .
\end{align*}   
  The underlying Hilbert space for our analysis is then given by
	\begin{equation}\label{HilbertSpace}
	L^2(\Graph)=\bigoplus_{\me \in \mE}L^2(0,\ell_\me)\ .
	\end{equation}
	We write $\mathcal{L}=\sum_{\me \in \mE} \ell_\me$ for the \emph{total length} of the graph. As in \cite{BK23}, we now introduce a Schrödinger operator on $L^2(\Graph)$ through
	\begin{equation}\label{SchrödingerOperator}
	-\Delta_\Graph+q\ ,
	\end{equation}
	which acts, on suitable functions $f=(f_\me)_{\me \in \mE} \in L^2(\Graph)$, via 
	\begin{equation}\label{ActionOperator}
	[\left(-\Delta_\Graph+q)f\right]_\me(x)=-f_\me^{\prime \prime}(x)+q_\me(x)f_\me(x)\ , \quad x \in (0,\ell_\me)\ .
	\end{equation}
	In this paper, we shall always assume that $q_\me \in (C^{\infty}\cap L^{\infty})(0,\ell_\me)$ for every $\me \in \mE$. This assumption guarantees that the operator $-\Delta_\Graph+q$ is in fact self-adjoint on a domain on which $-\Delta_\Graph$ is self-adjoint and, furthermore, it leads to the desired small-time asymptotics for the corresponding \emph{heat kernel} employing results of~\cite{BER}.
 \begin{remark}
     In \cite{BK23} the authors considered $\delta$-coupling conditions and, for the same reasons just mentioned, also assumed $q_\me \in (C^{\infty}\cap L^{\infty})(0,\ell_\me)$ for every $\me \in \mE$. However, as shown in the Appenidx~\ref{sec:general-potential-delta}, the results of \cite{BK23} can be extended to also hold for general potentials in $L^\infty(\Graph)$ using a bracketing argument for the corresponding $C_0$-semigroups. 
 \end{remark}
	In this paper we are interested in two particular self-adjoint domains of $-\Delta_\Graph$: the first one, for instance described in \cite[Proposition~2.1]{RS23}, is the so-called \emph{anti-Kirchhoff} realization given by
	\begin{equation*}
	\mathcal{D}_0:=\bigg\{f\in \widetilde{H}^2(\Graph): f'= (f_\me')_{\me \in \mE} \text{ is continuous at all vertices, and }  \sum_{\me \sim \mv} f_\me(\mv)=0  \bigg\}
	\end{equation*}
	where $\widetilde{H}^m(\Graph)$, $m \in \mathbb{N}$, denotes the subspace of $L^2(\Graph)$ consisting of all functions $f = (f_\me)_{\me \in \mE} \in L^2(\Graph)$ for which $f_\me \in H^m(0,\ell_\me)$ for all $\me \in \mE$, i.e.,
 \[
 \widetilde{H}^m(\Graph) := \bigoplus_{\me \in \mE} H^m(0,\ell_\me), \quad m \in \mathbb{N} \ .
 \]
 Note that we write $f_\me(\mv)$ for the boundary value of $f_\me \in H^1(0,\ell_\me)$ at either $x=0$ or $x=\ell_\me$, depending on the given vertex $\mv \in \mV$. Furthermore, the associated quadratic form is given by
 \begin{equation}
     h_0^q[f]=\int_{\mathcal{G}} \big(|f^{\prime}|^2 + q\vert f \vert^2 \big) \ \mathrm{d} x = \sum_{\me \in \mE} \int_0^{\ell_\me} \big( \vert f'_\me(x_\me) \vert^2 + q_\me(x_\me) \vert f_\me(x_\me) \vert^2 \big) \ \mathrm{d}x_\me
 \end{equation}
 with form domain 
 \begin{equation*}
     \cD[h_0^q]=\left\{f\in \widetilde{H}^1(\Graph):\ \mathcal{P}_{\mv,D}F(\mv)=0\ , \quad \forall \mv\in \mV \right\}\ ,
 \end{equation*}
 where $\mathcal{P}_{\mv,D} \in \mathbb{C}^{\deg(\mv) \times \deg(\mv)}$ is the (projection) matrix with constant entries $(\mathcal{P}_{\mv,D} )_{\me,\me'}=\frac{1}{\deg(\mv)}$ for all edges $\me,\me' \sim \mv$ and $F(\mv)=(f_{\me}(\mv))_{\me \sim \mv} \in \mathbb{C}^{\deg(\mv)}$ is the vector that collects all the boundary values of a function in a given vertex $\mv \in \mV$.

	The other self-adjoint realization of $-\Delta_\Graph$ -- referring to \emph{$\delta'$-coupling conditions} -- comes with the domain 
 \begin{equation*}\begin{split}
\mathcal{D}_{\beta}:=\bigg\{f\in \widetilde{H}^2(\Graph): f'= (f_\me')_{\me \in \mE} \text{ is continuous at all }& \text{vertices, and}   \sum_{\me \sim \mv} f_\me(\mv)=\beta_\mv f'(\mv)   \bigg\}  \ ,
    \end{split}
\end{equation*}
	where $\beta\equiv (\beta_\mv)_{\mv \in \mV} \in [0,\infty)^{|\mV|}$ collects the coupling strengths of the $\delta^{\prime}$-interactions in the vertices. 

 \begin{remark}
      Note that, since $f'$ is continuous at all vertices, i.e., $f_\me'(\mv) = f_{\me'}'(\mv)$ for any two edges $\me,\me' \in \mE$ connected to the same vertex $\mv \in \mV$,  the value $f'(\mv)$ is indeed well-defined.
 \end{remark}
 Moreover, whenever $\beta_v \neq 0$ for all $\mv \in \mV$, its associated quadratic form is given by
 \begin{equation}
    h_{\beta}^q[f]=\int_{\mathcal{G}}\big( |f^{\prime}|^2 + q\vert f \vert^2 \big)\ \mathrm{d}x+\sum_{\mv \in \mV} \frac{1}{\beta_{\mv}}\bigg|\sum_{\me \sim \mv} f_{\me}(\mv)
    \bigg|^2
 \end{equation}
 with form domain 
 \begin{equation}
     \cD[h_{\beta}^q]=\widetilde{H}^1(\Graph)=\left\{f\in \widetilde{H}^1(\Graph):\ \mathcal{P}_{\mv,\delta^{\prime}}F(\mv)=0\ , \quad \forall \mv\in \mV \right\}\ , 
 \end{equation}
 where $\mathcal{P}_{\mv,\delta^{\prime}}:=0 \in \mathbb{C}^{\deg(\mv) \times \deg(\mv)}$.
 \begin{remark}
     It is important to note that the anti-Kirchhoff $(-\Delta_\Graph, \mathcal{D}_0)$ realization is obtained from $(-\Delta_\Graph, \mathcal{D}_\beta)$ by choosing $\beta= (0,\dots,0)$. In this sense, the anti-Kirchhoff conditions are related to the $\delta^{\prime}$-coupling conditions in the same manner as the Kirchhoff conditions are to the $\delta$-coupling conditions. However, on a form level, the connection is less obvious since the form domains are not the same anymore. 
 \end{remark}
	Summing up, we are dealing with the self-adjoint (Schrödinger) operators $H^q_0=(-\Delta_\Graph+q,\mathcal{D}_{0})$ and $H^q_{\beta}=(-\Delta_\Graph+q,\mathcal{D}_{\beta})$. Any such operator has a purely discrete spectrum as the underlying metric graph is compact and finite. We shall denote the eigenvalues of $H^q_0$ as 
 $$\lambda^q_1(0) \leq \lambda^q_2(0) \leq \dots \rightarrow +\infty\ , $$
 and the eigenvalues of $H^q_{\beta}$ by
 $$\lambda^q_1(\beta) \leq \lambda^q_2(\beta) \leq \dots \rightarrow +\infty\ ,$$
 counting them with multiplicity in both cases. Also, $f^{0,q}_n$ and $f^{\beta,q}_n$, $n \in \mathbb{N}$, refer to corresponding orthogonal and normalized eigenfunctions.
		
	Finally, we may introduce the quantities we are going to study in the following, namely
\begin{equation}\begin{split}\label{DefinitionDifference}
d_n(\beta,\beta')&:=\lambda^q_n(\beta)-\lambda^q_n(\beta')\ , \quad n \in \mathbb{N}\ , \\
d^{q}_n(\beta,\beta')&:=\lambda^q_n(\beta)-\lambda^{q=0}_n(\beta')\ , \quad n \in \mathbb{N}\ .
\end{split}
\end{equation}
\section{Main spectral comparison results}\label{MainSpectral}
%
	%
 In this section we summarize our two main spectral comparison results whose proofs are then given in subsequent sections. The first result represents a counterpart to~\cite[Theorem~1]{BK23} and compares to different self-adjoint $\delta^{\prime}$-realizations with eath other. 
	\begin{theorem}[Comparison result $\delta^{\prime}$-conditions]\label{MainResult1} Let $\beta, \beta' \in (0,\infty)^{\vert \mV \vert}$ and $H^q_\beta, H^q_{\beta'}$ be two corresponding self-adjoint Schrödinger operators subject to $\delta'$-coupling conditions on a finite compact, connected quantum graph $\Graph$ as described above. Then one has
		\begin{equation}\label{RelationI}
		\lim_{N \rightarrow \infty}\frac{1}{N}\sum_{n=1}^{N} d_n(\beta,\beta')=\frac{2}{\mathcal{L}}\cdot \sum_{\mv \in \mV} \deg(\mv) \bigg( \frac{1}{\beta_\mv} - \frac{1}{\beta_\mv'} \bigg) \ .
		\end{equation}
  Furthermore, comparing the eigenvalues of $H^q_\beta, H^{q=0}_{\beta'}$ one obtains
  \begin{equation}\label{RelationPotential}
		\lim_{N \rightarrow \infty}\frac{1}{N}\sum_{n=1}^{N} d^q_n(\beta,\beta')=\frac{2}{\mathcal{L}}\cdot \sum_{\mv \in \mV} \deg(\mv) \bigg( \frac{1}{\beta_\mv} - \frac{1}{\beta_\mv'} \bigg)+\frac{1}{\mathcal{L}}\int_\Graph q\ \mathrm{d}x \ .
		\end{equation}
	\end{theorem}
\begin{remark}
  For constant parameters $\beta, \beta'$, formula~\eqref{RelationI} can be written as 
\begin{align}
\lim_{N \rightarrow \infty} \frac{1}{N} \sum_{n=1}^N d_n(\beta,\beta') = \frac{4\vert \mE \vert}{\mathcal{L}} \bigg( \frac{1}{\beta} - \frac{1}{\beta'} \bigg)\ ,
\end{align}
using the \emph{Handshake lemma} which states that $\sum_{\mv \in \mV} \deg(\mv) = 2 \vert \mE \vert$; of course, formula~\eqref{RelationPotential} can be rewritten accordingly.
\end{remark}
\begin{remark} Looking at a fixed graph $\Graph$ and setting $\beta^{\prime}=(2\gamma,\dots, 2\gamma)$ and $\beta=(\gamma,\dots, \gamma)$ for some $\gamma > 0$, one concludes that 
\begin{equation*}
    \lim_{N \rightarrow \infty}\frac{1}{N}\sum_{n=1}^{N} d_n(\beta,\beta') \sim \sum_{\mv \in \mV}\deg(\mv)
\end{equation*}
which means that the mean spectral shift is proportional to the volume of the associated \emph{discrete} graph. Interestingly, in the case of $\delta$-coupling conditions, the mean spectral shift is proprtional to the \emph{discrete surface measure} instead; see \cite{BK23,BKDiscrete}. 
\end{remark}

The second result aims at comparing $\delta^{\prime}$-conditions with anti-Kirchhoff conditions. Informally, since anti-Kirchhoff conditions are obtained by setting $\beta_{\mv}=0$ in a given vertex $\mv \in \mV$, a corresponding limit in Theorem~\ref{MainResult1} suggests that the mean eigenvalue differences should diverge in this situation. This is indeed supported by the following statement which also provides a rigorous explanation to a numerical observation made in \cite[Section~7.4]{BandSchanzSofer}.
\begin{theorem}[Comparison to anti-Kirchhoff conditions]\label{MainResultII}  Let $\Graph$ be a finite compact, connected quantum graph as described above. Consider the case where $\beta_\mv \equiv \beta \in [0,\infty)$ and $\beta_\mv' \equiv \beta' \in [0,\infty)$ for all $\mv \in \mV$ with $\beta, \beta' \in [0,\infty)$ be given such that $\beta + \beta' > 0$ and $\beta = 0$ or $\beta' = 0$. Moreover, let $H^q_\beta, H^q_{\beta'}$ be two corresponding self-adjoint Schrödinger operators either subject to anti-Kirchhoff or $\delta'$-coupling conditions. Then,
    \begin{equation}\label{RelationInfinite}
		\lim_{N \rightarrow \infty}\frac{1}{N}\sum_{n=1}^{N} d_n(\beta,\beta')=\lim_{N \rightarrow \infty}\frac{1}{N}\sum_{n=1}^{N} d_n^q(\beta,\beta') = \begin{cases}
		    +\infty \ , & \text{if $\beta = 0$}\ , \\ -\infty\ , & \text{if $\beta' = 0$\ .}
		\end{cases}
		\end{equation}
\end{theorem}

	\section{A modified local Weyl law for $\delta'$-coupling conditions}
	This section provides a crucial ingredient -- a so-called local Weyl law -- in order to show Theorem~\ref{MainResult1}. One should stress that local Weyl laws are interesting in their own right. For example, as shown in~\cite{BKInfinite} considering infinite metric graphs, modified local Weyl laws lead to different spectral comparison results. Also, compared to the local Weyl law given in \cite[Proposition~4]{BK23} for $\delta$-coupling conditions (see also \cite[Theorem~4.1]{BHJ} for a special case), we now establish a \emph{modified} local Weyl law also for $\delta^{\prime}$-coupling conditions. 
	\begin{proposition}[Local Weyl law for $\delta'$-coupling conditions]\label{prop:localweyllaw}
		Let $\Graph$ be a graph with properties described above and let $v \in \mV$ be some vertex. Then
		\[
		\lim_{N \rightarrow \infty} \frac{1}{N} \sum_{n=1}^N \bigg\vert \sum_{\me \sim \mv} (f^{\beta,q}_n)_\me(\mv) \bigg\vert^2 = \frac{2\deg(\mv)}{\mathcal{L}}\ ,
		\]
		holds for any $\beta \in \mathbb{R}^{\vert \mV \vert}$ where $\beta_\mv \neq 0$. Furthermore, for any interior point $x \in (0,\ell_\me)$ one has
\begin{equation*}
     \lim_{N \rightarrow \infty}\frac{1}{N} \sum_{n=1}^N\big\vert f^{\beta, q}_n(x) \big\vert^2=\frac{1}{\mathcal{L}}\ .
\end{equation*}
	\end{proposition}
\begin{proof}
	According to \cite[Proposition~8.1]{BER}, the generated semigroup $(\mathrm{e}^{-t H_\beta^q})_{t \geq 0}$ consists of integral operators $\mathrm{e}^{-tH_\beta^q}$ on $L^2(\Graph)$ induced by an integral kernel $p^{H_\beta^q}(t;\cdot, \cdot) \in L^\infty(\Graph \times \Graph)$ which is (jointly) smooth edgewise. 
 
 We now prove the first part of the theorem: Due to \cite[Proposition~8.1]{BER}, it follows that, $x \simeq (\mv,\me)$ and $y \simeq (\mv,\me')$,
	\begin{align}\label{eq:bolte-egger-rueckriemen-short-time-heat-kernel}
	p^{H_\beta^q}_{\me,\me'}(t;x,y) \sim \frac{\delta_{\me,\me'} + S_\mv(\infty)_{\me,\me'}}{\sqrt{4\pi t}} \ , \quad \text{as } t \rightarrow 0^+, 
	\end{align}
	where $S_\mv(\infty)= \mathrm{Id}-2\mathcal{P}_{\mv,\beta}$. Thus, \eqref{eq:bolte-egger-rueckriemen-short-time-heat-kernel} readily implies that
	\begin{align*}
	p^{H_\beta^q}_{\me,\me'}(t;x,y) \sim \begin{cases} \frac{2}{\sqrt{4\pi t}}\ , & \text{if } \me = \me'\ , \\ 0\ , & \text{if } \me \neq \me'\ , \end{cases} \quad \text{as }t \rightarrow 0^+\ .
	\end{align*}
	Then, defining the positive measure 
\[
\mu := \sum_{n=1}^\infty \bigg\vert \sum_{\me \sim \mv} (f^{\beta,q}_n)_\me(\mv) \bigg\vert^2 \delta_{\lambda_n^q(\beta)} \ ,
\]	
	where $\delta_{\lambda_n^q(\beta)}$ denotes the Dirac measure, Mercer's theorem yields
	\begin{align}
	\begin{aligned}
	\int_0^\infty \mathrm{e}^{-tx} \mathrm{d}\mu(x) &= \sum_{n=1}^\infty \mathrm{e}^{-\lambda_n^q(\beta)t} \bigg\vert \sum_{\me \sim \mv} (f^{\beta,q}_n)_\me(\mv) \bigg\vert^2 \\&= \sum_{\me,\me' \sim \mv} \sum_{n=1}^\infty \mathrm{e}^{-\lambda_n^q(\beta)t} (f_n^{\beta,q})_\me(\mv) (f_n^{\beta,q})_{\me'}(\mv) \\&= \sum_{\me,\me' \sim \mv} p^{H_\beta^q}_{\me,\me'}(t;x,y) = \sum_{\stackrel{\me,\me' \sim \mv}{\me \neq \me'}} p^{H_\beta^q}_{\me,\me'}(t;x,y) + \sum_{\me \sim \mv} p^{H_\beta^q}_{\me,\me}(t;x,y) \\&\sim 0 + \sum_{\me \sim \mv} \frac{2}{\sqrt{4\pi t}} = \frac{2\deg(\mv)}{\sqrt{4\pi t}}, \quad \text{as } t \rightarrow 0^+ \ .
	\end{aligned}
	\end{align}
	According to the Karamata's theorem (see, for instance, \cite[Theorem~6.33]{BorSpectralTheory}), this implies
	\[
	\mu[0,\lambda] = \sum_{\lambda_n^{q}(\beta) \leq \lambda} \bigg\vert \sum_{\me \sim \mv} (f^{\beta,q}_n)_\me(\mv) \bigg\vert^2 \sim \frac{2\deg(\mv)}{\pi} \sqrt{\lambda} \ , \quad \text{as } \lambda \rightarrow \infty\ .
	\]
	Employing the classical Weyl law $\lambda_N^q(\beta) \sim (\pi N / \mathcal{L})^2$ to replace the summation condition $\lambda_n^q(\beta) \leq \lambda$ by $1 \leq n \leq N$ then yields
	\begin{align*}
	 \sum_{n=1}^N \bigg\vert \sum_{\me \sim \mv} (f^{\beta,q}_n)_\me(\mv) \bigg\vert^2 \sim \frac{2\deg(\mv)}{\pi} \frac{\pi N}{\mathcal{L}} \:\: \Longleftrightarrow \:\: \frac{1}{N}\sum_{n=1}^N \bigg\vert \sum_{\me \sim \mv} (f^{\beta,q}_n)_\me(\mv) \bigg\vert^2 \sim \frac{2\deg(\mv)}{\mathcal{L}} \ ,
	\end{align*}
	as $N \rightarrow \infty$.

The second part of the theorem follows from an application of \cite[Remark~5]{BK23}, imagining the presence of a so-called ``dummy vertex'' at $x \in (0,\ell_\me)$ at which one imposes natural boundary conditions, i.e., one assumes continuity of the function as well as the derivatives. 
\end{proof}

\section{Proofs of the main spectral comparison results}\label{ProofSection}
In this section we establish the proofs of the main spectral comparison results, Theorem~\ref{MainResult1} and Theorem~\ref{MainResultII}.
\subsection{Proof of Theorem~\ref{MainResult1}} The proof follows the methods developed in \cite[Section~3]{BK23}, using Proposition \ref{prop:localweyllaw} and a version of \cite[Corollary~6.5]{RS23}. 

    In a first step we write, using a Hellmann-Feynman or Hadamard-type formula~\cite{LSHadamard} in combination with the chain rule, we obtain
    \begin{equation}
        d_n(\beta,\beta') = \int_{0}^{1}\sum_{\mv \in \mV}(\beta_\mv'-\beta_\mv)\left(\frac{1}{(\beta'_\mv+\tau(\beta_\mv-\beta'_\mv)}\right)^2\bigg\vert \sum_{\me \sim \mv} 
        (f^{\beta_\mv'+\tau(\beta_\mv - \beta_\mv'),q}_n)_\me(\mv) \bigg\vert^2\ \mathrm{d}\tau\ ,
    \end{equation}
    and
    \begin{equation}
    \begin{split}
        d_n^q(\beta,\beta') &= \sum_{\me \in \mE}\int_{0}^{\ell_\me}   q_\me(x) \big\vert f^{\beta_\mv'+\tau(\beta_\mv - \beta_\mv'),\tau q}_n(x) \big\vert^2 \mathrm{d}x \\
        &\quad + \int_{0}^{1} \sum_{\mv \in \mV}(\beta^{\prime}_{\mv}-\beta_{\mv})\left(\frac{1}{(\beta^{\prime}_\mv+\tau(\beta_\mv-\beta'_{\mv})}\right)^2\bigg\vert \sum_{\me \sim \mv} (f^{\beta_\mv'+\tau(\beta_\mv - \beta_\mv'),\tau q}_n)_{\me}(\mv) \bigg\vert^2\ \mathrm{d}\tau  \ .
        \end{split}
    \end{equation}
Then, using a version of \cite[Lemma~1.3]{BK23}, one can observe that there exists some constant $C > 0$ such that
\[
\Vert f_n^{\beta_\mv' + \tau(\beta_\mv-\beta_\mv'),q} \Vert_{L^\infty(\Graph)} \leq C
\]
as well as 
\[
\Vert f_n^{\beta_\mv' + \tau(\beta_\mv-\beta_\mv'),\tau q} \Vert_{L^\infty(\Graph)} \leq C
\]
holds for all $n \in \mathbb{N}$ and all $\tau \in [0,1]$ (in other words, the eigenfunctions are \emph{uniformly} bounded in supremum norm). Thus, by dominated convergence and Proposition~\ref{prop:localweyllaw}, one obtains
\begin{equation*}\begin{split}
&\lim_{N \rightarrow \infty}\frac{1}{N}\sum_{n=1}^{N} d_n(\beta,\beta')
\\& \qquad\quad= \int_{0}^{1}\sum_{\mv \in \mV}\frac{\beta_\mv'-\beta_\mv}{(\beta'_\mv+\tau(\beta_\mv-\beta'_\mv))^2} \bigg(\lim_{N \rightarrow \infty}\frac{1}{N} \sum_{n=1}^N \bigg\vert \sum_{\me \sim \mv} (f^{\beta_\mv'+\tau(\beta_\mv - \beta_\mv'),q}_n)_\me(\mv)\bigg\vert^2 \bigg) \, \mathrm{d}\tau \\& \qquad\quad= \frac{2}{\mathcal{L}} \sum_{\mv \in \mV} \deg(\mv)\int_{0}^{1}\sum_{\mv \in \mV}\frac{\beta_\mv'-\beta_\mv}{(\beta'_\mv+\tau(\beta_\mv-\beta'_\mv))^2} \, \mathrm{d}\tau = \frac{2}{\mathcal{L}} \sum_{\mv \in \mV} \deg(\mv)\left(\frac{1}{\beta_\mv}-\frac{1}{\beta'_\mv}\right). 
\end{split}
\end{equation*}
Similarly,  
\begin{equation*}\begin{split}
\lim_{N \rightarrow \infty}\frac{1}{N}\sum_{n=1}^{N} d^q_n(\beta,\beta')
&=  \frac{2}{\mathcal{L}} \sum_{\mv \in \mV} \deg(\mv)\left(\frac{1}{\beta_\mv}-\frac{1}{\beta'_\mv}\right) \\
& \qquad \quad+  \sum_{\me \in \mE}\int_{0}^{\ell_\me}   q_\me(x) \left(\lim_{N \rightarrow \infty}\frac{1}{N} \sum_{n=1}^N\big\vert f^{\beta_\mv'+\tau(\beta_\mv - \beta_\mv'),\tau q}_n(x) \big\vert^2\right) \mathrm{d}x
\\& =\frac{2}{\mathcal{L}}\cdot \sum_{\mv \in \mV} \deg(\mv) \bigg( \frac{1}{\beta_\mv} - \frac{1}{\beta_\mv'} \bigg)+\frac{1}{\mathcal{L}}\int_\Graph q\ \mathrm{d}x\ .
\end{split}
\end{equation*}
\subsection{Proof of Theorem \ref{MainResultII}} This theorem follows from a direct application of Theorem~\ref{MainResult1} as well as the min-max principle. Recall that the $n$-th eigenvalue is characterized by
\begin{align}\label{eq:min-max}
\lambda_n^q(\beta) = \min_{\stackrel{\mathcal{F} \subset \mathcal{D}[h_\beta^q] \text{ subspace}}{\dim(\mathcal{F}) = n}} \max_{0 \neq f \in \mathcal{F}} \frac{h_\beta^q[f]}{\int_\Graph \vert f \vert^2 \ \mathrm{d}x}, \quad n \in \mathbb{N}\ .
\end{align}
    Now, since $h_\beta^q[f] = h^q_0[f]$ for any $f \in \mathcal{D}[h^q_0] \subset \mathcal{D}[h_\beta^q]$, \eqref{eq:min-max} implies
    \[
    \lambda_n^q(\beta) \leq \lambda_n^q(0)  \quad \text{for all $n \in \mathbb{N}$ and all $\beta \in [0,\infty)$}\ ;
    \]
    this should be compared with \cite[Corollary~5.3.(d)]{RS23}. According to Theorem \ref{MainResult1}, this implies, for fixed $\beta,\beta' \in [0,\infty)$, that
    \begin{align*}
         \frac{1}{N} \sum_{n=1}^N d_n(0,\beta') &= \frac{1}{N} \sum_{n=1}^N \big(\lambda_n(0) - \lambda_n(\beta')\big) \\&\geq \frac{1}{N} \sum_{n=1}^N \big(\lambda_n(\gamma) - \lambda_n(\beta')\big) \quad \text{for all $\gamma \in (0,\infty)$ and all $N \in \mathbb{N}$}\ ,
    \end{align*}
    as well as 
    \begin{align*}
         \frac{1}{N} \sum_{n=1}^N d_n(\beta,0) &= \frac{1}{N} \sum_{n=1}^N \big(\lambda_n(\beta) - \lambda_n(0)\big) \\&\leq \frac{1}{N} \sum_{n=1}^N \big(\lambda_n(\beta) - \lambda_n(\gamma)\big) \quad \text{for all $\gamma \in (0,\infty)$ and all $N \in \mathbb{N}$}\ .
    \end{align*}
    Due to Theorem~\ref{MainResult1}, whenever $\beta^{\prime} \neq 0$, one has
    \begin{align}\label{eq:limsup}
        \liminf_{N \rightarrow \infty} \frac{1}{N} \sum_{n=1}^N d_n(0,\beta') \geq \frac{4\vert \mE \vert}{\mathcal{L}} \bigg( \frac{1}{\gamma} - \frac{1}{\beta'} \bigg) \quad \text{for every $\gamma \in (0,\infty)$} \ ,
    \end{align}
    and, whenever $\beta \neq 0$, 
    \begin{align}\label{eq:liminf}
        \limsup_{N \rightarrow \infty} \frac{1}{N} \sum_{n=1}^N d_n(\beta,0) \leq \frac{4\vert \mE \vert}{\mathcal{L}} \bigg( \frac{1}{\beta} - \frac{1}{\gamma} \bigg) \quad \text{for every $\gamma \in (0,\infty)$} \ .
    \end{align}
   The statement then follows by considering the limit $\gamma \rightarrow 0^+$. Also, the same argument can now be repeated to obtain, for any $\gamma \in (0,\infty)$ and whenever $\beta^{\prime} \neq 0$,
    \begin{align*}\label{eq:limsup2}
        \liminf_{N \rightarrow \infty} \frac{1}{N} \sum_{n=1}^N d_n^q(0,\beta') \geq \frac{4\vert \mE \vert}{\mathcal{L}} \bigg( \frac{1}{\gamma} - \frac{1}{\beta'} \bigg) + \frac{1}{\mathcal{L}} \int_\Graph q \ \mathrm{d}x \ ,
    \end{align*}
    and, for any $\gamma \in (0,\infty)$ and whenever $\beta \neq 0$,
    \begin{equation}\label{eq:liminf2}
        \limsup_{N \rightarrow \infty} \frac{1}{N} \sum_{n=1}^N d_n(\beta,0) \leq \frac{4\vert \mE \vert}{\mathcal{L}} \bigg( \frac{1}{\beta'} - \frac{1}{\gamma} \bigg) + \frac{1}{\mathcal{L}} \int_\Graph q \ \mathrm{d}x \ .
    \end{equation}
   The statement then again follows by considering the limit $\gamma \rightarrow 0^+$.
\begin{example}[Spectral comparison on intervals]
    Assume that $\mathcal{G} \simeq [0,\mathcal{L}]$ is an interval of length $\mathcal{L} > 0$. Then, according to Theorem \ref{MainResultII}, we have that
\begin{equation*}
    \lim_{N \rightarrow \infty}\frac{1}{N}\sum_{n=1}^{N}d_n(0,\beta)=\infty\ .
\end{equation*}
Note that the case $\beta = 0$ on an interval simply corresponds to Dirichlet boundary conditions at the interval ends. In addition, $\delta^{\prime}$-conditions with coupling strength $\beta > 0$ correspond to $\delta$-coupling conditions with strength $\frac{1}{\beta} > 0$ (compare with Section~\ref{SomeAdditionalSpectral}). Consequently, we conclude that Dirichlet boundary conditions -- at least on an interval-- are not asymptotically close to neither $\delta^{\prime}$- nor $\delta$-coupling conditions.
\end{example}

\section{Some additional spectral comparison result}\label{SomeAdditionalSpectral}

In this section, we want to comment on some additional spectral comparison results. More explicitly, one may ask what happens if one compares the eigenvalues of a Schrödinger operator subject to $\delta$-coupling conditions to the ones of a Schrödinger operator subject to $\delta^{\prime}$-coupling conditions. 

In order to (partly) answer this question, we first introduce $\delta$-coupling conditions in more detail: Assume that $\Graph$ is a graph with the same properties as assumed above. Let then $\sigma = (\sigma_\mv)_{\mv \in \mV} \in \mathbb{R}^{\vert \mV \vert}$ be given and denote (with a slight abuse of notation) by $H_{\sigma}^q$ the unique self-adjoint operator associated with the quadratic form
\begin{equation}\label{FormDelta}
    h_\sigma^q[f] := \int_\Graph \big( \vert f' \vert^2 + q\vert f \vert^2 \big)\ \mathrm{d}x + \sum_{\mv \in \mV} \sigma_\mv \vert f(\mv) \vert^2
\end{equation}
with form domain
\begin{equation}\label{FormDeltaDomain}
\mathcal{D}[h_\sigma^q] := H^1(\Graph) := C(\Graph) \cap \widetilde{H}^1(\Graph)\ ,
\end{equation}
where $C(\Graph)$ denotes the space of \emph{continuous} functions on $\Graph$, cf.\ also \cite[Equation~(8)]{BK23}. As an operator, $H_{\sigma}^q$ acts as in~\eqref{ActionOperator} with operator domain
 \begin{equation*}\begin{split}
\mathcal{D}_{\sigma}:=\bigg\{f\in \widetilde{H}^2(\Graph): f= (f_\me)_{\me \in \mE} \text{ is continuous at all }& \text{vertices, and}   \sum_{\me \sim \mv} f^{\prime}_\me(\mv)=\sigma_\mv f(\mv)   \bigg\}  \ ,
    \end{split}
\end{equation*}
 where $f^{\prime}_\me(\mv)$ refers to the \emph{inward} derivative at the vertex $\mv \in \mV$. In our setting, $H^{q}_{\sigma}$ has purely discrete spectrum as well and we shall denote is eigenvalues by $(\mu_n^q(\sigma))_{n \in \mathbb{N}}$. We then set 
\begin{align*}
    d_n^{\delta,\delta'}(\sigma,\beta) := \mu_n^q(\sigma) - \lambda_n^q(\beta), \quad n \in \mathbb{N} \ ,
\end{align*}
where $(\lambda_n^q(\sigma))_{n \in \mathbb{N}}$ are the eigenvalues of a Schrödinger operator $H^{q}_{\beta}$ on the same graph $\Graph$ but subject to $\delta^{\prime}$-coupling conditions. 

 Before we state our results, recall that a graph $\Graph$ is called \emph{bipartite} if one can write $\mV = \mV_1 \cup \mV_2$ for disjoint nonempty subsets $\mV_1, \mV_2 \subset \mV$ such that each edge has an endpoint corresponding to a vertex in $\mV_1$ and one endpoint corresponding to a vertex in $\mV_2$. In addition, the \emph{Betti number} of a graph is defined by $\beta(\Graph) := \vert \mE \vert - \vert \mV \vert + 1$ equal to $1$ (cf.\ also Figure \ref{fig:bipartite}).
\begin{figure}[h]
\begin{tikzpicture}[scale=0.55]
      \tikzset{enclosed/.style={draw, circle, inner sep=0pt, minimum size=.1cm, fill=gray}, every loop/.style={}, every fit/.style={ellipse,draw,inner sep=-2.5pt,text width=2cm, line width=1pt}}

      \node[enclosed] (A) at (0,0) {};
      \node[enclosed] (B) at (0,2) {};
      \node[enclosed] (C) at (0,4) {};
      \node[enclosed] (D) at (0,6) {};
      \node[enclosed] (B') at (8,1) {};
      \node[enclosed] (C') at (8,3) {};
      \node[enclosed] (D') at (8,5) {};
      
      \node [gray,fit=(A) (D),label=left:\textcolor{black}{$\mV_1 \quad$}] {};
      \node [gray,fit=(B') (D'),label=right:\textcolor{black}{$\quad \mV_2$}] {};
      \draw (A) edge node[above] {} (B') node[midway, above] (edge2) {};
      \draw (A) edge node[above] {} (C') node[midway, above] (edge3) {};
      \draw (B) edge node[above] {} (B') node[midway, above] (edge2) {};
      \draw (B) edge node[above] {} (D') node[midway, above] (edge4) {};
      \draw (C) edge node[above] {} (B') node[midway, above] (edge2) {};
      \draw (C) edge node[above] {} (C') node[midway, above] (edge3) {};
      \draw (D) edge node[above] {} (C') node[midway, above] (edge3) {};
     \end{tikzpicture}
     \caption{A bipartite graph with vertex set $\mV = \mV_1 \cup \mV_1$ and $\beta(\Graph) = 1$.}\label{fig:bipartite}
     \end{figure}
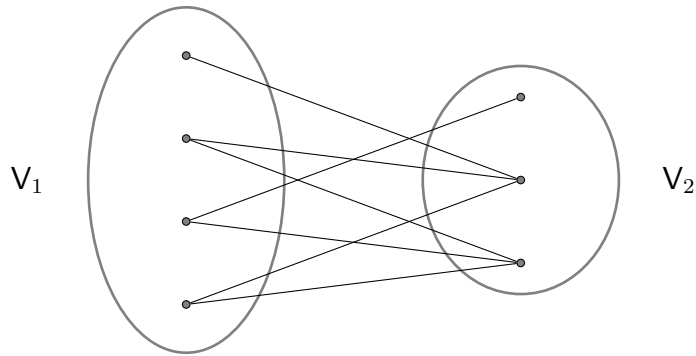

Now, in \cite[Corollary~3.7]{KurRoh} it was shown that $H^{q=0}_{\beta=0}$ and $H^{q=0}_{\sigma=0}$ are isospectral if and only if $\Graph$ is bipartite with $\beta(\Graph)=1$. Using this, Theorem \ref{MainResultII} implies the following spectral comparison result.
\begin{corollary}[Spectral comparison result for bipartite graphs]\label{cor:mixed-spectral-comparison} Let $\Graph$ be a finite compact, connected and bipartite quantum graph with $\beta(\Graph) = 1$. Furthermore, let $\beta \in (0,\infty)^{\vert \mV \vert}$ and $\sigma \in [0,\infty)^{\vert \mV \vert}$ be given with associated Schrödinger operators $H_\beta^q$ and $H_\sigma^q$. Then
\begin{equation}\label{RelationInfiniteMixed}
	\lim_{N \rightarrow \infty}\frac{1}{N}\sum_{n=1}^{N} d_n^{\delta,\delta'}(\sigma,\beta) = +\infty.
	\end{equation}
\end{corollary}
\begin{proof}
   First, due to the min-max principle, one has $\mu_n^q(\sigma) \geq \mu_n^{q = 0}(\sigma) - \Vert q \Vert_\infty$ and likewise $\mu_n^q(\beta) \leq \mu_n^{q = 0}(\beta) + \Vert q \Vert_\infty$ for every $n \in \mathbb{N}$. With this we conclude
   \begin{align*}
       \frac{1}{N} \sum_{n=1}^N d_n^{\delta,\delta'}(\sigma,\beta) &\geq \frac{1}{N} \sum_{n=1}^N \big(\mu_n^{q=0}(0) - \lambda_n^{q=0}(\beta) \big) - 2\Vert q \Vert_\infty \\&= \frac{1}{N} \sum_{n=1}^N \big(\lambda_n^{q=0}(0) - \lambda_n^{q=0}(\beta) \big) - 2\Vert q \Vert_\infty \\&\geq \frac{1}{N} \sum_{n=1}^N \big(\lambda_n^{q=0}(0) - \lambda_n^{q=0}(\beta') \big) - 2\Vert q \Vert_\infty, \quad \text{for every $N \in \mathbb{N}$} \ ,
   \end{align*}
   where $\beta' := \big(\min_{\mv \in \mV} \beta_\mv \big) \cdot \mathbf{1} \in (0,\infty)^{\vert \mV \vert}$ (and $\mathbf{1}$ is the vector consisting of ones only), and since $\mu_n^{q=0}(0) = \lambda_n^{q=0}(0)$ for all $n \in \mathbb{N}$ due to \cite[Corollary~3.7]{KurRoh}. Theorem \ref{MainResultII} then implies
   \begin{align*}
       \liminf_{N \rightarrow \infty} \frac{1}{N} \sum_{n=1}^N d_n^{\delta,\delta'}(\sigma,\beta) \geq \lim_{N \rightarrow \infty } \frac{1}{N}\sum_{n=1}^N \big(\lambda_n^{q=0}(0) - \lambda_n^{q=0}(\beta') \big) - 2\Vert q \Vert_\infty = +\infty \ ,
   \end{align*}
   which completes the proof.
\end{proof}

\appendix 

\section{Small-time asymptotics of the heat-kernel for bounded potentials}\label{sec:general-potential-delta}
In \cite{BER}, when studying properties of the heat-kernel on graphs, the authors restrict attention to bounded potentials in $L^\infty(\Graph)$ that are also smooth on any edge (in other words, they assume the same regularity as we do in this paper). Since the results of~\cite{BER} were employed in \cite{BK23, BKAmba, BKInfinite}, the same assumption on the potential had been imposed. However, in this appendix we want to show, using methods from the theory of strongly continuous semigroups, that for the small-time asymptotics it is enough to only assume boundedness of the potential, at least as long as one focusses on $\delta$-coupling conditions. Unfortunately, however, the methods presented here do not allow to also cover $\delta'$-coupling conditions since our proof heavily relies on the positivity of the underlying generated strongly continuous semigroup.

 Regarding the associated heat-kernel of a Schrödinger operator $H^{q}_{\sigma}$ (see Section~\ref{SomeAdditionalSpectral}), one again has $p^{H_{\sigma}^q}(t;\cdot,\cdot) \in L^\infty(\Graph \times \Graph)$, and more precisely, $p^{H_{\sigma}^q}(t;\cdot,\cdot)  \in C^{0,1}(\Graph \times \Graph)$, where $C^{0,1}(\Graph \times \Graph)$ denotes the space of \emph{jointly Lipschitz continuous} functions on $\Graph \times \Graph$, cf.\ \cite{KostenkoMugnoloNicolussi, BifulcoMugnolo}.

For bounded potentials, we then establish the following result which forms a generalization of \cite[Proposition~8.1]{BER} for $\delta$-conditions.
\begin{theorem}\label{thm:general-potentials}
    Assume that $\Graph$ is a graph as above and let $q \in L^\infty(\Graph)$ and $\sigma \in \mathbb{R}^{\vert \mV \vert}$ be given. Then, for $x \in \Graph$ (understanding that $\deg(x)=2$ whenever $x$ in an interior point), 
    \begin{align}\label{eq:short-time-asymptotic-heat-kernel-general-potentials}
        p_{\me,\me'}^{H_\sigma^q}(t;x,x) \sim \frac{\delta_{\me,\me'} + T_\mv(\infty)_{\me,\me'}}{\sqrt{4\pi t}} = \frac{2}{\deg(x)} \frac{1}{\sqrt{4\pi t}}\ , \quad \text{as $t \rightarrow 0^+$}\ ,
    \end{align}
    where $T_\mv(\infty):=\mathrm{Id}-2\mathcal{P}_{\mv,\delta}$ with 
     \begin{align*}
    \mathcal{P}_{\mv,\delta}= \begin{pmatrix}
	1-\frac{1}{\deg(\mv)} & -\frac{1}{\deg(\mv)} & \dots & -\frac{1}{\deg(\mv)} \\ -\frac{1}{\deg(\mv)} & \ddots & \ddots &  \vdots \\ \vdots & \ddots & \ddots & -\frac{1}{\deg(\mv)} \\ -\frac{1}{\deg(\mv)} & \dots & -\frac{1}{\deg(\mv)} & 1-\frac{1}{\deg(\mv)}
	\end{pmatrix} \in \mathbb{C}^{\deg(\mv) \times \deg(\mv)} \ . 
	\end{align*}
\end{theorem}

The basic proof strategy follows~\cite[Lemma~11]{BKInfinite}.

\begin{proof}
   Define
    \[
    q_- := -\Vert q \Vert_{L^\infty(\Graph)} \cdot \mathbf{1} \quad \text{and} \quad q_+ := \Vert q \Vert_{L^\infty(\Graph)} \cdot \mathbf{1}\ ,
    \]
    where $\mathbf{1}(x)=1$ is the constant function. Then $q_-,q_+ \in L^\infty(\Graph) \cap \bigoplus_{\me \in \mE} C^\infty(0,\ell_\me)$ with
    \[
    \cD[h_\sigma^q] = \cD[h_\sigma^{q_-}] = \cD[h_\sigma^{q_+}] = H^1(\Graph)
    \]
    as well as, for the associated sesquilinear forms, $h_\sigma^{q_-}[f,g] \leq h_\sigma^q[f,g] \leq h_\sigma^{q_+}[f,g]$ for all $0 \leq f,g \in H^1(\Graph)$. Since all underlying semigroups $(\mathrm{e}^{-tH_{\sigma}^q})_{t \geq 0}$, $(\mathrm{e}^{-tH_{\sigma}^{q_-}})_{t \geq 0}$ and $(\mathrm{e}^{-tH_{\sigma}^{q_+}})_{t \geq 0}$ are positive due to the Beurling--Deny--Ouhabaz criterion for positive $C_0$-semigroups (see, for instance, \cite[Proposition~5.5]{ArendtGlueck20}), one concludes by \cite[Theorem~2.24]{Ouh05} that
    \begin{align*}
        \big\vert \mathrm{e}^{-tH_{\sigma}^{q_+}} f \big\vert \leq \mathrm{e}^{-tH_{\sigma}^q} \vert f \vert \quad \text{and} \quad \big\vert \mathrm{e}^{-tH_{\sigma}^q} f \big\vert \leq \mathrm{e}^{-tH_{\sigma}^{q_-}} \vert f \vert \quad \text{for all $f \in L^2(\Graph)$}
    \end{align*}
    implying that
    \begin{align}\label{eq:bracketing}
        p_{\me,\me'}^{H_\sigma^{q+}}(t;x,y) \leq p_{\me,\me'}^{H_\sigma^{q}}(t;x,y)  \leq p_{\me,\me'}^{H_\sigma^{q-}}(t;x,y) \quad \text{for all $t > 0$ and all $x,y \in \Graph$} \ .
    \end{align}
    Now, as both $q_-$ and $q_+$ are smooth on each edge, the asymptotics of \eqref{eq:short-time-asymptotic-heat-kernel-general-potentials} hold according to \cite[Proposition~8.1]{BER}. The statement then follows from~\eqref{eq:bracketing}.
\end{proof}
Using the asymptotics of~\eqref{eq:short-time-asymptotic-heat-kernel-general-potentials} one can deduce, for example, a more general version of \cite[Theorem~1]{BK23}. The same is true for corresponding results established in \cite{BKAmba,BKInfinite}.

\subsection*{Acknowledgement}{PB was supported by the Deutsche Forschungsgemeinschaft DFG (Grant 397230547). JK appreciates interesting discussions with G.~Berkolaiko (Texas) as well as M.~Hofmann (Hagen).}

\end{document}